\documentclass[11pt]{amsart} 
\usepackage[left=1.5in,right=1.5in,top=1.15in,bottom=1.15in,includehead,includefoot]{geometry}
\usepackage{amssymb,graphicx}
\usepackage{enumitem}
\usepackage[pdftitle={Tree-width and dimension},
            pdfauthor={G. Joret, P. Micek, K. G. Milans, W. T. Trotter, B. Walczak, R. Wang},
            colorlinks=true,linkcolor=black,citecolor=black,filecolor=black,urlcolor=black]{hyperref}

\newtheorem{theorem}{Theorem}[section]
\newtheorem{conjecture}[theorem]{Conjecture}
\newtheorem{corollary}[theorem]{Corollary}
\newtheorem{lemma}[theorem]{Lemma}
\newtheorem{exercise}[theorem]{Exercise}
\theoremstyle{definition}
\newtheorem{question}[theorem]{Question}

\newcommand{\cgF}{\mathcal{F}}
\newcommand{\cgG}{\mathcal{G}}
\newcommand{\cgR}{\mathcal{R}}
\newcommand{\Inc}{\operatorname{Inc}}
\newcommand{\pval}{\operatorname{val}}

\newcommand{\close}[1]{\mathrel{\rhd_{#1}}}

\def\figurescale{.6}

\renewenvironment{enumerate}{\begin{enumorig}[label=\textup{(\arabic*)}, noitemsep, topsep=1mm, labelindent=.5em, leftmargin=*]}{\end{enumorig}}

\makeatletter
\let\old@setaddresses\@setaddresses
\def\@setaddresses{\bigskip\bgroup\parindent 0pt\let\scshape\relax\old@setaddresses\egroup}
\makeatother

\title{Tree-width and dimension}

\author[G.~Joret\and P.~Micek\and K.~G.~Milans\and W.~T.~Trotter\and B.~Walczak\and R.~Wang]{Gwena\"{e}l Joret\and Piotr Micek\and Kevin G. Milans\and William T. Trotter\and Bartosz Walczak\and Ruidong Wang}

\thanks{Bartosz Walczak was supported by Polish National Science Center
grant 2011/03/N/ST6/03111.}

\address[G.~Joret]{Department of Mathematics and Statistics\\
The University of Melbourne\\Melbourne\\Australia}
\email{gjoret@ulb.ac.be}

\address[P.~Micek, B.~Walczak]{Theoretical Computer Science Department\\Faculty of Mathematics and
Computer Science\\Jagiellonian University\\Krak\'ow\\Poland}
\email{micek@tcs.uj.edu.pl, walczak@tcs.uj.edu.pl}

\address[K.~G.~Milans]{Department of Mathematics\\West Virginia University\\Morgantown, West
Virginia 26505\\USA}
\email{milans@math.wvu.edu}

\address[W.~T.~Trotter, R.~Wang]{School of Mathematics\\Georgia Institute of Technology\\Atlanta,
Georgia 30332\\USA}
\email{trotter@math.gatech.edu, rwang49@math.gatech.edu}

\begin{document}

\begin{abstract} 
Over the last 30 years, researchers have investigated
connections between dimension for posets and planarity for graphs. Here we
extend this line of research to the structural graph theory parameter tree-width
by proving that the dimension of a finite poset is bounded in terms of its
height and the tree-width of its cover graph.
\end{abstract}

\maketitle

\section{Introduction}

In this paper, we investigate combinatorial problems involving finite graphs and
partially ordered sets (posets), linking the well-studied concept of tree-width
for graphs with the concept of dimension for posets. The following is
our main result.

\begin{theorem}\label{thm:main} 
For every pair\/ $(t,h)$ of positive integers,
there exists a least positive integer\/ $d = d(t,h)$ so that if\/ $P$ is a poset of
height at most\/ $h$ and the tree-width of the cover graph of\/ $P$ is at most\/ $t$,
then the dimension of\/ $P$ is at most\/~$d$. In particular, we have\/ $d(t,h) \le
6\cdot 2^{8t^{4h-2}}$.
\end{theorem}

The remainder of this paper is organized as follows. In the next section, we
provide a brief summary of essential notation and terminology for posets and
dimension. This is followed by an even more compact section on graphs and
tree-width. These sections are included since we anticipate that many readers
will be quite familiar with one of these topics but less so with the other. 
With these basics in hand, we discuss in Section~\ref{sec:motivation} the background
behind this line of research and the motivation for our principal theorem. The
proof of our main theorem is given in Section~\ref{sec:proof}, and we discuss
some open problems in Section~\ref{sec:questions}.

\section{Posets and Dimension}

A \textit{partially ordered set} (here we use the short term \textit{poset}) is
a set $P$ equipped with a reflexive, antisymmetric and transitive binary
relation $\le$. Elements of $P$ are called points and here we will also call
them vertices, since we will often consider graphs whose vertex set is the set
of elements of $P$. When the poset $P$ is fixed throughout the discussion, we
abbreviate the statement $x\le y$ in $P$ by just writing $x\le y$. The notation
$x<y$ means of course $x\le y$ and $x\neq y$. These notations are reversible in
the obvious manner, i.e., $x>y$ means the same as $y<x$. 

We say $x$ \textit{covers} $y$ (also $y$ is 
\textit{covered by} $x$) when $x>y$, and there is no point $z$ with $x>z>y$. 
Also, we associate with a poset $P$ a \textit{cover graph} having
the same vertex set as $P$.  The cover graph of $P$ has
an edge $xy$ when one of $x$ and $y$ covers the other.
A drawing (typically, we consider only drawings with straight line 
segments for the edges) of the cover graph of a poset $P$ is called 
an \textit{order diagram} (also, a \textit{Hasse diagram}) if the point 
in the plane corresponding to the point $x$ is higher than the point 
corresponding to the point $y$ when $x$ covers $y$ in
$P$. We show in Figure~\ref{fig:same-cover-graph} order 
diagrams for two different posets, both with the same cover graph.

\begin{figure} 
\begin{center} 
\includegraphics[scale=\figurescale]{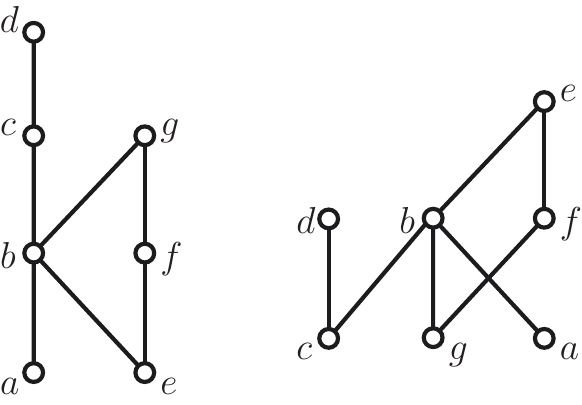}
\end{center} 
\caption{Two posets with the same cover graph
\label{fig:same-cover-graph}} 
\end{figure}

When $x$ and $y$ are distinct points in a poset $P$, and either $x<y$ or $y<x$, 
we say $x$ and $y$ are \textit{comparable}. When $x$ and $y$ 
are distinct points in $P$, and they are not comparable, we say they 
are \textit{incomparable} and write $x\parallel y$ . We use the notation $\Inc(P)$ 
for the set of all ordered pairs $(x,y)$ with $x\parallel y$. 

An element $a$ in a poset $P$ is \textit{minimal}, respectively \textit{maximal}
when there is no point $x$ with $x<a$, respectively $x>a$. When $Q$ is a subset
of a poset $P$, the restriction of the binary relation $\le$ to $Q$ is a poset
and we call this a \textit{subposet} of $P$. A poset $P$ is called a
\textit{linear order} (also a \textit{total order}) if $\Inc(P)=\emptyset$. When
$Q$ is a subposet of $P$ and $Q$ is a linear order, it is customary to call $Q$
a \textit{chain}. The largest positive integer $h$ for which $P$ has a subposet
$Q$ on $h$ points which is a chain in $P$ is called the \textit{height} of $P$.

A poset $P$ is called an \textit{antichain} if it has height~$1$, i.e., $x\parallel
y$ for all $x$ and $y$ with $x\neq y$. The largest integer $w$ for which $P$
contains a subposet on $w$ points which is an antichain is called the
\textit{width} of $P$. The classic theorem of Dilworth~\cite{bib:Dilworth}
asserts that a poset of width $w$ can be partitioned into
$w$ chains. Dually, Mirsky~\cite{bib:Mirsky} proved that a poset of height~$h$
can be partitioned into $h$ antichains.

Let $P$ and $L$ be posets. We call $L$ a \textit{linear extension} of $P$ when 
$L$ has the same ground set as $P$, $L$ is a linear order, and $x<y$ in $L$ whenever
$x<y$ in $P$. A family $\cgR=\{L_1,L_2,\dots,L_t\}$ of linear extensions of $P$
is called a \textit{realizer} of $P$ if $x<y$ in $P$ if and only if $x<y$ in
$L_i$ for each $i=1,2,\dots,t$. Clearly, a family $\cgR=\{L_1,L_2,\dots,L_t\}$
of linear extensions of $P$ is a realizer if and only if for each
$(x,y)\in\Inc(P)$, there is some $i$ with $1\le i\le t$ such that $x>y$ in
$L_i$.

Dushnik and Miller~\cite{bib:Dushnik-Miller} defined the \textit{dimension} of
$P$, denoted $\dim(P)$, as the least positive integer $t$ for which $P$ has a
realizer $\cgR$ with $|\cgR|=t$. Evidently, $\dim(P)=1$ if and only if $P$ is a
linear order. Also, when $P$ is a non-trivial antichain, $\dim(P)=2$ as
evidenced by the realizer $\{L, L^d\}$ where $L$ is an arbitrary linear order on
the ground set of $P$ and $L^d$ is the \textit{dual} of $L$, i.e., $x>y$ in
$L^d$ if and only if $x<y$ in $L$.

In~\cite{bib:Hiraguchi}, Hiraguchi used Dilworth's theorem to show that the
dimension of a poset never exceeds its width. Hiraguchi also proved that if $P$
is a poset on $n$ points with $n\ge4$, then $\dim(P)\le\lfloor n/2\rfloor$. Both
these inequalities are tight, as witnessed by a family of posets called
\textit{standard examples} and first studied in~\cite{bib:Dushnik-Miller}. As
these posets play an important role later in this paper, we include here some
details on their structure and properties.
 
For $d\ge2$, the standard example $S_d$ is a height $2$ poset with minimal
elements $\{a_1,a_2,\dots,a_d\}$ and maximal elements $\{b_1, b_2,\dots,b_d\}$.
The relation $\le$ is defined on $S_d$ by setting $a_i<b_j$ if and only if
$i\neq j$, for all $i,j=1,2,\dots,d$. For each $d\ge2$, the width of $S_d$ is
$d$ so $\dim(S_d)\le d$. On the other hand, $\dim(S_d)\ge d$. This follows from
the observation that if $L$ is a linear extension of $S_d$, there can only be
one integer $i$ with $1\le i\le d$ and $a_i>b_i$ in $L$. Moreover, when $d\ge3$,
it is easy to see that $S_d$ is \textit{$d$-irreducible}, i.e., removing any
point from $S_d$ lowers the dimension to $d-1$.

There is a natural notion of \textit{isomorphism} for posets, and it obvious
that isomorphic posets have the same dimension. So it is natural to say that a
poset $P$ \textit{contains} a poset $Q$ when there is a subposet of $P$ which
is isomorphic to $Q$. In this vein, a poset $P$ has large dimension when it
contains a large standard example. But this is far from necessary.

A poset $P$ is called an \textit{interval order} when there is a family
$\{[a_x,b_x]: x\in P\}$ of closed intervals of the real line $\mathbb{R}$ so
that $x<y$ in $P$ if and only if $b_x<a_y$ in $\mathbb{R}$.
Fishburn~\cite{bib:Fishburn} showed that a poset $P$ is an interval order if and
only if it does not contain the standard example $S_2$.
In~\cite{bib:Furedi-Hajnal-Rodl-Trotter}, F\"{u}redi, Hajnal, R\"{o}dl and Trotter
show that the maximum dimension of an interval order $P$ of height~$h$ is
$\lg\lg h + (1/2+o(1))\lg\lg\lg h$. In particular, note that in order for an
interval order to have large dimension, it must have very large height.

The standard examples show that in general, large height is not necessary for
large dimension, and in~\cite{bib:Felsner-Li-Trotter}, Felsner, Li and Trotter
show that for every pair $(g,d)$ of positive integers, there is a height~$2$
poset $P$ with $\dim(P)\ge d$ so that the girth of the cover graph of $P$ is at
least $g$. The posets resulting from this construction contain $S_2$ but they do
not contain $S_3$, when $g>6$.

Although cover graphs are useful in providing diagrams of posets, they
do not seem to tell us much about the combinatorial properties of the
posets associated with them.  For example, the two posets shown
in Figure~\ref{fig:same-cover-graph} have the same
cover graph.  However, the poset on the left has height~$4$, width~$2$ and
$21$ linear extensions, while the poset on the right has height~$3$, width~$3$
and $84$ linear extensions.  Both posets have dimension~$2$.

At the extreme, a linear order on
$n$ points has height~$n$, width~$1$ and of course, a unique linear
extension.  However, when $n\ge2$, the associated cover graph is bipartite, and the
height~$2$ poset with this same cover graph is called a \textit{fence}.  
Now the width is $\lceil n/2\rceil$ and the number of linear extensions
is exponentially large in $n$.  On the other hand, the dimension of a fence
is~$2$ when $n\ge3$, so based only on these observations, one might conjecture 
that posets with the same cover graph have approximately the same dimension.  
But even this is not true.  Later in the paper, we will show that for each 
$d\ge1$, there are two posets having the same cover graph, one having 
dimension~$2$ and the other having dimension at least~$d$.

However, there is another natural way to associate a graph with a
poset.  Like the cover graph, the \textit{comparability graph} of $P$
has the same vertex set as $P$ but now we make $xy$ an edge if $x$ and
$y$ are comparable.  The comparability graph of a poset contains the
cover graph as a subgraph.  Furthermore, if $P$ and $Q$ are posets with 
isomorphic comparability graphs, then they have the same height, width, 
number of linear extensions and dimension.  The fact that they have the same 
height and width is immediate.  The fact that they have the same number of linear 
extensions and the same dimension follows in a straightforward manner from 
the pioneering work of Gallai~\cite{bib:Gallai} on comparability graphs.

With these remarks in mind, and with no additional background information
to suggest otherwise, the principal result of this paper would then have to be 
viewed as a surprise.

\section{Graphs and Tree-Width}\label{s:tree-width}

In this paper, we consider only finite graphs without loops or multiple edges,
and we assume that readers are familiar with basic concepts such as trees,
paths, cycles, complete graphs, subgraphs, induced subgraphs, components,
chromatic number, girth, genus, distance and diameter. Given a graph $G$, an
induced subgraph of $G$ is determined entirely by its vertex set. In particular,
when $T$ is a tree, we will identify subtrees of $T$ just by specifying their
vertex sets. So when $T'$ and $T''$ are subtrees of a tree $T$, the statement
$T'\cap T''\neq\emptyset$ just means that $T'$ and $T''$ have one or more
vertices of $T$ in common.

Let $G$ be a graph with vertex set $V$ and edge set $E$. The
\textit{tree-width}\footnote{We refer the reader to the text by
Diestel~\cite{bib:Diestel} for a concise exposition of some of the key concepts
behind this parameter. Diestel also provides interesting details on its history
and the twenty year time period spanned by Robertson and Seymour's proof of the
Graph Minor Theorem. Also our notation for tree-width and some of our examples
are taken from exercises in this text.} of $G$ is the least positive integer $t$
for which there is a tree $T$ and a family $\{T(x): x\in V\}$ of non-empty
subtrees of $T$ so that
\begin{enumerate}
\item for all vertices $u$ in $T$, $|\{x\in V:u\in
T(x)\}|\le t+1$,
\item $T(x)\cap T(y)\neq\emptyset$ for all $xy\in E$.
\end{enumerate}

Trivially, a graph has tree-width~$0$ if and only if it has no edges, while
the tree-width of the complete graph $K_n$ on $n$ vertices is $n-1$ for
all $n\ge1$.  Furthermore,
if $G=(V,E)$ is a tree with at least one edge, then the tree-width 
of $G$ is~$1$. To see this, simply
subdivide each edge $e=xy$ in $E$ by inserting a new vertex $m_{xy}$ in the
interior of~$e$. Let $T$ denote the resulting tree. Then for each $x\in V$, take
$T(x)$ as the subtree of $T$ with vertex set $\{x\}\cup\{m_{xy}:xy\in E\}$ (each
$T(x)$ is a star). Conversely, it is easy to see that a graph $G$ has
tree-width at most~$1$ if and only if it is acyclic.

Consider the following three basic operations on a graph: (1)~delete an edge;
(2)~delete a vertex; (3)~contract an edge. Given a graph $G$, any graph $H$ that
can be obtained from $G$ by applying a sequence of these basic operations is
called a \textit{minor} of $G$. The following fundamentally important theorem,
called the Graph Minor Theorem, is due to Robertson and
Seymour~\cite{bib:Robertson-Seymour}\footnote{The proof given by 
Robertson and Seymour for the Graph Minor Theorem appears in a series 
of papers published over the time span 1983 through 2004, and we cite 
here the culminating paper in that series.}.

\begin{theorem}\label{thm:Robertson-Seymour} 
If\/ $\{G_n:n\ge 1\}$ is an infinite sequence of graphs, then there are integers\/ 
$i$ and\/ $j$ with\/ $1\le i<j$ so that\/ $G_i$ is isomorphic to a minor of\/ $G_j$. 
\end{theorem}

A class $\cgG$ of graphs is \textit{minor-closed} if $H$ is in $\cgG$ whenever
$G$ is in $\cgG$ and $H$ is isomorphic to a minor of $G$. Examples of minor
closed classes of graphs include the family of all planar graphs and, more
generally, for fixed $g\ge0$, the family of all graphs having genus at most~$g$.
Also, it is easy to see that for each $t\ge1$, the class of all graphs having
tree-width at most~$t$ is minor-closed.

Any proper minor-closed class of graphs admits a characterization by
``forbidden minors'', i.e., a minimum family $\cgF$ of graphs such that a graph
$G$ belongs to $\cgG$ if and only if it does not contain a minor isomorphic to a
graph in $\cgF$. By the Graph Minor Theorem, the class $\cgF$ is finite. The
classic theorem of Wagner~\cite{bib:Wagner} 
asserts that the list of forbidden minors for the
class of planar graphs consists of the complete graph $K_5$ and the complete
bipartite graph $K_{3,3}$.

Planar graphs can have large tree-width.  Note that any bipartite graph
is both the cover graph and the comparability graph of a height~$2$
poset.  In particular, the $n\times n$ planar grid is bipartite and
has tree-width $n$ (see Diestel~\cite{bib:Diestel}, Exercises~14 and~21
on page~369). However, the tree-width of a planar graph is bounded
in terms of its diameter\footnote{This result is implicit in the work of
Baker~\cite{bib:Baker} and made explicit by Bodlaender in~\cite{bib:Bodlaender}.}.
Classes of graphs where tree-width is bounded in terms of diameter are said to
satisfy the \textit{diameter tree-width property} (also called the
\textit{bounded local tree-width property}). 

The concept of \textit{path-width} for graphs is defined just like tree-width
except that it is required that the tree $T$ be a path, and of course the
subtrees of $T$ are then just subpaths of $T$. Trivially, the tree-width of a
graph is at most its path-width. However, the tree-width of an outerplanar
graph is at most~$2$ (this follows from the observation that in a maximal
outerplanar graph, there is always a vertex $x$ of degree two such that the
neighbors of $x$ are adjacent to each other). On the other hand, outerplanar
graphs can have arbitrarily large path-width. In fact, trees 
can have arbitrarily large path-width (see Diestel~\cite{bib:Diestel}, Exercise~31
on page 370).

\section{Background and Motivation}\label{sec:motivation}

A poset $P$ is \textit{planar} if its order diagram can be drawn without edge
crossings in the plane. In Figure~\ref{fig:non-planar-poset}, we show on
the left the order diagram of a height~$3$ nonplanar poset.  However,
the cover graph of this poset is planar as witnessed by the drawing on the
right. 

\begin{figure} 
\begin{center} 
\includegraphics[scale=\figurescale]{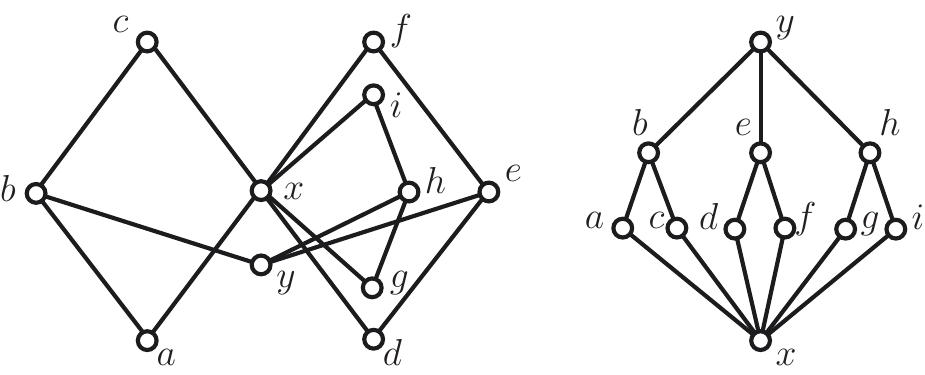}
\end{center} 
\caption{A non-planar poset with planar cover graph
\label{fig:non-planar-poset}}
\end{figure}

We note that if $P$ is a
height~$2$ poset, then $P$ is planar if and only if its cover graph is
planar~\cite{bib:Moore, bib:DiBattista-Liu-Rival}. We also note that it
is NP-complete to test whether a poset is planar~\cite{bib:Garg-Tamassia}, while
there are linear-time algorithms for testing whether a graph is
planar~\cite{bib:Hopcroft-Tarjan}. Also, it is NP-complete to test whether a
graph is a cover graph~\cite{bib:Brightwell, bib:Nesetril-Rodl}.

When $P$ is a poset with only one minimal element, this single element is usually
called a \textit{zero}.  Similarly, in a poset with only one maximal element,
this element is called a \textit{one}.  The first result linking planarity 
and dimension is the following theorem of Baker, Fishburn and 
Roberts~\cite{bib:Baker-Fishburn-Roberts}.

\begin{theorem}\label{thm:Baker-Fishburn-Roberts}
If\/ $P$ is a planar poset with a zero and a one, then\/
$\dim(P)\le 2$.
\end{theorem}

Subsequently, Trotter and Moore~\cite{bib:Trotter-Moore} proved the following
extension.

\begin{theorem}\label{thm:Trotter-Moore}
If\/ $P$ is a planar poset with a zero or a one, then\/
$\dim(P)\le 3$.
\end{theorem}

Trotter and Moore~\cite{bib:Trotter-Moore} also obtained the
following result as an immediate corollary to the preceding theorem.

\begin{corollary}\label{cor:Trotter-Moore}
If\/ $P$ is a poset whose cover graph is a tree, then\/ $\dim(P)\le 3$.
\end{corollary}

With the benefit of hindsight, one can argue that the line of research
carried out in this paper might reasonably have been triggered 35 years 
ago, based solely on possible extensions to Corollary~\ref{cor:Trotter-Moore}.

It is an easy exercise to show that the standard example $S_d$ is planar when
$d\le 4$, and as a consequence, there are $4$-dimensional planar posets.
On the other hand,  $S_d$ is non-planar when $d\ge5$. 
For a brief time in the late 1970's, it was thought that it might be the 
case that $\dim(P)\le4$ whenever $P$ is a planar poset.

However, in 1981, Kelly~\cite{bib:Kelly} showed that for each $d\ge5$, the
standard example $S_d$ is a subposet of a planar poset $P_d$. We illustrate
Kelly's construction in Figure~\ref{fig:Kelly} for the specific value $d=6$.

\begin{figure} 
\begin{center} 
\includegraphics[scale=\figurescale]{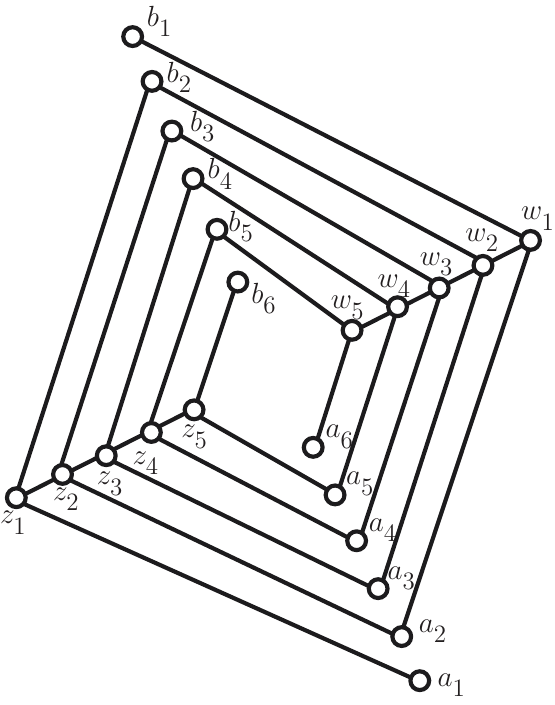}
\end{center}
\caption{Kelly's construction\label{fig:Kelly}}
\end{figure}

We pause here to answer a question raised earlier concerning the
dimension of posets with the same cover graph. Specifically, we show
that for each $d\ge2$, there are posets $Q_d$ and $Q'_d$ with the same
cover graph with $\dim(Q'_d)=2$ and $\dim(Q_d)\ge d$.  
First, we consider a poset $Q_d$ formed by modifying Kelly's
example as follows.  For each $i=1,2,\dots,d-1$, we
add two new minimal points $u_i$ and $v_i$ with
$u_i$ covered by $b_i$ and $b_{i+1}$, while $v_i$ is covered by
$a_i$ and $a_{i+1}$.  Clearly, $P_d$ is a subposet of $Q_d$ so that 
$\dim(Q_d)\ge d$. 

On the other hand, there are exponentially many posets having the same cover 
graph as $Q_d$.  One of them, which we denote $Q'_d$, has $b_i>u_i>b_{i+1}$ and 
$a_{i+1}>v_i>a_i$, for each $i=1,2,\dots,d-1$.  Obviously, both
$Q_d$ and $Q'_d$ are planar poset as witnessed by trivial modifications
to the diagram for $P_d$ given in Figure~\ref{fig:Kelly}.  Moreover, in
$Q'_d$, the point $a_1$ is now a zero and the point $b_d$ is now a one.
So by Theorem~\ref{thm:Baker-Fishburn-Roberts},
$\dim(Q'_d)=2$.

Returning to the general subject of the dimension of posets with
planar cover graphs, Felsner, Li and Trotter~\cite{bib:Felsner-Li-Trotter} 
proved the following result in 2010:

\begin{theorem}\label{thm:Felsner-Li-Trotter} 
Let\/ $P$ be poset of height\/~$2$. If the cover graph of\/ $P$ is planar, then\/ 
$\dim(P)\le4$. 
\end{theorem}

Actually, this was obtained as an easy corollary to the
following theorems of Brightwell and Trotter~\cite{bib:Brightwell-Trotter-2,
bib:Brightwell-Trotter-1}, published in 1997 and 1993, respectively (a new
and quite elegant proof of this result has just been obtained by
Felsner~\cite{bib:Felsner}).

\begin{theorem}\label{thm:Brightwell-Trotter}
Let\/ $G$ be a planar multi-graph and let\/ $P$ be the vertex-edge-face
poset determined by a drawing without edge crossings of\/ $G$ in the plane.
Then\/ $\dim(P)\le4$. Furthermore, if\/ $G$ is a simple, $3$-connected planar
graph, then the subposet determined by the vertices and faces is\/
$4$-irreducible.
\end{theorem}

The inequality in Theorem \ref{thm:Felsner-Li-Trotter} is best possible as
evidenced by the standard example $S_4$. Noting that the poset $P_d$ in 
Kelly's construction has height $d+1$, Felsner, Li and 
Trotter~\cite{bib:Felsner-Li-Trotter} conjectured
the following generalization, which was proved by Streib and
Trotter~\cite{bib:Streib-Trotter} in 2012.

\begin{theorem}\label{thm:Streib-Trotter}
For every positive integer\/ $h$, there is a least positive integer\/ $c_h$
so that if\/ $P$ is a poset with a planar cover graph and the height of\/ $P$
is at most\/ $h$, then\/ $\dim(P)\le c_h$.
\end{theorem}

We have $c_1=2$ and $c_2=4$. For $h\ge 3$, the upper bound on the constant
$c_h$ produced in the proof of Theorem \ref{thm:Streib-Trotter}
is very large, as several iterations of Ramsey theory are used. From
below, it is straightforward to modify Kelly's original construction and
decrease the height to $d-1$. This can be accomplished by deleting $a_1$, $a_d$,
$b_1$ and $b_d$ and relabelling $z_1$, $z_{d-1}$, $w_1$ and $w_{d-1}$ as $a_1$,
$b_d$, $b_1$ and $a_d$, respectively. Wiechert~\cite{bib:Wiechert}
constructed a planar poset $Q$ of height~$3$ with $\dim(Q)=5$; however, this
construction does not seem to generalize for larger values of $d$. Accordingly,
when $d\ge6$, we do not know whether there is a planar poset $P$ of height
$d-2$ with $\dim(P)=d$. On the other hand, Streib and
Trotter~\cite{bib:Streib-Trotter} showed that for each $d\ge5$, there is a poset
$P$ of height $d-2$ with $\dim(P)=d$ so that the cover graph of $P$ is planar.

Theorems~\ref{thm:Baker-Fishburn-Roberts} and~\ref{thm:Trotter-Moore},
as well as Corollary~\ref{cor:Trotter-Moore} all
provide conditions where the dimension of a planar poset can be bounded
independent of its height.  In~\cite{bib:Felsner-Trotter-Wiechert},
Felsner, Trotter and Wiechert gave the following additional results of
this nature.

\begin{theorem}\label{thm:Felsner-Trotter-Wiechert} 
Let\/ $P$ be a poset. 
\begin{enumerate}
\item If the cover graph of\/ $P$ is outerplanar, then\/ $\dim(P)\le 4$. 
\item If the comparability graph of\/ $P$ is planar, then\/ $\dim(P)\le 4$.
\end{enumerate}
\end{theorem}

Both inequalities in Theorem~\ref{thm:Felsner-Trotter-Wiechert} are 
best possible.  The proof of the first inequality in 
Theorem~\ref{thm:Felsner-Trotter-Wiechert} is
relatively straightforward, but it takes a bit of work to show that it is best
possible.  However, the second inequality in 
Theorem~\ref{thm:Felsner-Trotter-Wiechert} is quite different, and now
the argument depends on the full strength of the Brightwell-Trotter 
inequality for the dimension of the vertex-edge-face poset determined 
by a drawing of a planar multi-graph, with the edges now playing a key role.

To the best of our knowledge, the following observation concerning Kelly's 1981
construction was not made until 2012:\enspace The cover graphs of the posets in
this construction have bounded tree-width. In fact, they have bounded
path-width.  We leave the following elementary observations as an exercise.

\begin{exercise}\label{exe:kelly-pw} 
Let\/ $d\ge 2$, let\/ $P_d$ be the poset illustrated in Kelly's construction, and 
let\/ $G_d$ be the cover graph of\/ $P_d$.  Then the height of\/ $P_d$ is\/ $d+1$, and
the path-width of\/ $G_d$ is at most\/~$3$. In fact, when\/ $d\ge5$, $G_d$ 
contains\/ $K_4$ as a minor, so its path-width is exactly\/~$3$. 
\end{exercise}

We made some effort to construct large dimension posets with bounded height and
cover graphs having bounded tree-width and were unable to do so. So consider the
following additional information:

\begin{enumerate} 
\item A poset whose cover graph has tree-width $1$ has dimension at most~$3$. 
\item A poset whose cover graph is outerplanar has dimension at most~$4$. As 
noted previously, outerplanar graphs can have arbitrarily large path-width, 
but they have tree-width at most~$2$. 
\item On the one hand, the tree-width of the cover graph of a planar poset 
can be arbitrarily large, even when the height of $P$ is~$2$. As an 
example, just take a height~$2$ poset whose cover graph is an $n\times n$ 
grid. On the other hand, the proof given by Streib and Trotter%
~\cite{bib:Streib-Trotter} to show that the dimension of a poset with a 
planar cover graph can be bounded in terms of its height used
a reduction to the case where the cover graph of the poset is both planar and
has diameter bounded in terms of the height of the poset. Again, as noted
previously, a planar graph of bounded diameter has bounded tree-width.
\end{enumerate}

Taking into consideration this body of evidence together with our inability to
prove otherwise, it is natural to conjecture that the dimension of a poset is
bounded in terms of its height and the tree-width of its cover graph, and this
is what we now prove.

\section{Proof of the Main Theorem}\label{sec:proof}

\subsection{Preliminaries} 

A subset $I$ of $\Inc(P)$ is said to be
\textit{reversible} if there is a linear extension $L$ of $P$ with $x>y$ in $L$
for every $(x,y)\in I$. It is then immediate that $\dim(P)$ is the least positive
integer $d$ so that there is a partition $\Inc(P)=I_1\cup I_2\cup\dots\cup I_d$
with each $I_i$ reversible. In view of this formulation, it is handy to have a
simple test to determine whether a given subset $I$ of $\Inc(P)$ is reversible.

Let $k\ge 2$. An indexed subset $I=\{(x_i,y_i):1\le i\le k\}$ of $\Inc(P)$ is
called an \textit{alternating cycle} when $x_i\le y_{i+1}$ in $P$ for each
$i\in\{1,2,\dots,k\}$, where we interpret the subscripts cyclically (i.e., we
require $x_k\le y_1$ in $P$). Reversing an alternating cycle $I$ would require a
linear extension in which the cyclic arrangement $y_1, x_1, \ldots, y_k, x_k$
alternates between strict inequalities of the form $y_i < x_i$ (needed to reverse
$I$) and inequalities of the form $x_i \le y_{i+1}$ (forced by $P$).
Consequently, alternating cycles are not reversible. The following elementary
lemma, proved by Trotter and Moore in \cite{bib:Trotter-Moore} using slightly
different terminology, states that alternating cycles are the only obstruction
to being reversible.

\begin{lemma}\label{lem:reversible} 
If\/ $P$ is a poset and\/ $I\subseteq\Inc(P)$,
then\/ $I$ is reversible if and only if\/ $I$ contains no alternating cycle.
\end{lemma}

For the remainder of this section, we fix integers $t$ and $h$, assume that $P$
is a poset with height $h$ and cover graph $G$, and assume that the tree-width
of $G$ is $t$. Of course, we may also assume that $\Inc(P)\neq \emptyset$. The
remainder of the argument is organized to show that we can partition the set
$\Inc(P)$ into $d$ reversible sets, where $d$ is bounded in terms of $t$ and
$h$.

Let $X$ denote the ground set of $P$, so that $X$ is also the vertex set of the
cover graph $G$. Since the tree-width of $G$ is $t$, there is a tree $T$ and a
family $\cgF=\{T(x):x\in X\}$ of subtrees of $T$ such that~(1)~for each vertex
$u$ of $T$, the number of elements $x$ of $X$ with $u\in T(x)$ is at most $t+1$,
and (2)~for each edge $xy$ of $G$, we have $T(x)\cap T(y)\ne\emptyset$.

Let $H$ be the intersection graph determined by the family $\cgF$ of subtrees
of $T$ (some researchers refer to $H$ as the \textit{chordal completion} of $G$).
Evidently, the tree-width of $H$ is $t$, and every edge of $G$ is an edge of
$H$. Of course, the set $X$ is also the vertex set of $H$. In the discussion to
follow, we will go back and forth, without further comment, between referring to
members of $X$ as elements of the poset $P$ and as vertices in the cover graph
$G$ and the intersection graph $H$.

To help distinguish between vertices of $T$ and elements of $X$, we will use the
letters $r$, $u$, $v$ and $w$ (possibly with subscripts) to denote vertices of
the tree $T$, while the letters $x$, $y$ and $z$ (again with subscripts) will be
used to denote members of $X$. The letters $i$, $j$, $k$, $\ell$, $m$ and $n$
will denote non-negative integers with the meaning of $n$ fixed by setting
$n=|X|$. The Greek letters $\phi$ and $\tau$ will denote proper colorings of the
graph $H$. The colors assigned by $\phi$ will be positive integers, while the
colors assigned by $\tau$ will be sets of triples. Later, we will define a
function $\sigma$ which assigns to each incomparable pair $(x,y)$ a
\textit{signature}, to be denoted $\sigma(x,y)$. We will use the Greek letter
$\Sigma$ to denote a signature. The number of signatures will be the value $d$,
and we will use Lemma~\ref{lem:reversible} to show that any set of incomparable
pairs having the same signature is reversible. Of course, we must be careful to
insure that $d$ is bounded in terms of $t$ and $h$.

We consider the tree $T$ as a \textit{rooted} tree by taking an arbitrary vertex
$u_0$ of $T$ as root. Draw the tree without edge crossings in the canonical
manner. The root is at the bottom, and each vertex that is not the root has a
unique neighbor below---its parent (equipped with such a drawing, $T$ is called
a \textit{planted tree}). We suggest such a drawing in 
Figure~\ref{fig:planted-tree}.

\begin{figure}
\begin{center}
\includegraphics[scale=\figurescale]{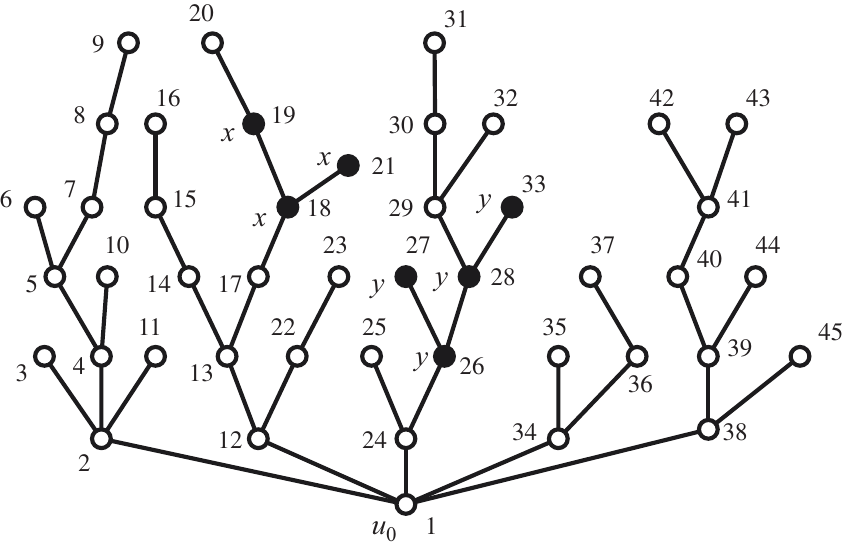}
\end{center}
\caption{A tree with root $u_0$ and vertices labelled using number
$1,2,\dots,45$ according to the depth-first, left-to-right search order $L_1$.
Two subtrees $T(x)=\{18,19,21\}$ and $T(y)=\{26,27,28,33\}$ are marked with the
darkened points. The root $r(x)$ of $T(x)$ is $18$, and the root $r(y)$ of
$T(y)$ is $26$.
\label{fig:planted-tree}}
\end{figure}

For each $x\in X$, let $r(x)$ denote the root of the subtree $T(x)$, i.e., the
unique vertex of $T(x)$ that is closest to the root $u_0$ of $T$. Expanding
vertices of $T$ if necessary, we may assume that $r(x)\ne r(y)$ whenever $x$ and
$y$ are distinct elements of $X$.

The tree $T$ may be considered as a poset by setting $u\le v$ in $T$ when $u$
lies on the path from $v$ to $u_0$ in $T$. Let $L_1$ denote the depth-first,
left-to-right search order of $T$. Let $L_2$ denote the depth-first,
right-to-left search order of $T$. It follows that $u\le v$ in $T$ if and only
if $u\le v$ in $L_1$ and $u\le v$ in $L_2$\footnote{Note that the poset obtained
by adding a one to $T$ is planar.  Now the argument given 
in~\cite{bib:Baker-Fishburn-Roberts} implies that $\dim(T)\le2$, as evidenced
by these two linear extensions.}. This shows $\dim(T)\le 2$ with
$\dim(T)=2$ unless $L_1=L_2$. It is natural to say that $u$ is \textit{left} of
$v$ in $T$, when $u<v$ in $L_1$ and $v<u$ in $L_2$. Also, we say that $u$ is
\textit{below} $v$ in $T$ when $u<v$ in $T$. When $u$ and $v$ are distinct
elements of $T$, exactly one of the following four statements holds: (1)~$u$ is
below $v$ in $T$; (2)~$v$ is below $u$ in $T$; (3)~$u$ is left of $v$ in $T$;
and (4)~$v$ is left of $u$ in $T$.

The \textit{lowest common ancestor} of two vertices $u$ and $v$ of $T$, 
denoted $u\wedge v$, is the greatest vertex $w$ with $w\le u$ and $w\le v$ in $T$.

\subsection{Induced Paths in the Intersection Graph}

Observe that $xy$ is an edge of the graph $H$ if and only if one of the
following statements is true: (1)~$r(x)<r(y)$ in $T$ and $r(y)\in T(x)$,
(2)~$r(y)<r(x)$ in $T$ and $r(x)\in T(y)$.

We write $x\close{k}y$ when there is a sequence $(z_0,z_1,\dots,z_m)$ of
elements of $X$ such that $0\le m\le k$, $z_0=x$, $z_m=y$, and $r(z_i)\in
T(z_{i+1})$ for each $i\in\{0,1,\dots,m-1\}$. Note that a shortest such sequence
is an induced path in the graph $H$. Therefore, we could alternatively have
written this definition as follows: $x\close{k}y$ when there is an induced path
$(z_0,z_1,\dots,z_m)$ in $H$ with $0\le m\le k$, $z_0=x$, $z_m=y$, and
$r(z_0)>r(z_1)>\dots>r(z_m)$ in $T$. As it will turn out, our proof will use the
relation $\close{k}$ for $k\le 2h-2$.

\begin{lemma}\label{lem:close}
The relation\/ $\close{k}$ has the following
properties:
\begin{enumerate}
\item if\/ $x\close{k}y$ and\/ $k\le\ell$, then\/ $x\close{\ell}y$, 
\item $x\close{0}y$ if and only if\/ $x=y$, 
\item $x\close{k+\ell}z$ if and only if there exists\/ $y\in X$ with\/ $x\close{k}y$ and\/
$y\close{\ell}z$, 
\item if\/ $x\close{k}y$, then\/ $r(y)\le r(x)$ in\/ $T$, 
\item if\/ $x\close{k}z$ and\/ $r(z)\le r(y)\le r(x)$ in\/ $T$, then\/ $y\close{k}z$, 
\item if\/ $x\close{k}y$ and\/ $x\close{k}z$, then\/ $y\close{k}z$ or\/ $z\close{k}y$, 
\item for each\/ $x\in X$, $|\{y\in X:x\close{k}y\}|\le 1+t+t^2+\dots+t^k$.
\end{enumerate}
\end{lemma}

\begin{proof}
Properties (1)--(4) follow directly from the definition of
$\close{k}$. To see (5), let $(z_0,z_1,\dots,z_m)$ be a sequence of elements of
$X$ such that $0\le m\le k$, $z_0=x$, $z_m=z$, and $r(z_i)\in T(z_{i+1})$ for
each $i\in\{0,1,\ldots,m-1\}$. Now note that since $(z_0,\ldots,z_m)$ is a path
in $H$ and $r(x) = r(z_0) \in T(z_1)$, the union $\bigcup_{i=1}^{m} T(z_i)$ is a
subtree of $T$ containing the path from $r(x)$ to $r(z)$. In particular,
$r(y)\in \bigcup_{i=1}^{m} T(z_i)$, so there must be a positive $i$ with
$r(y)\in T(z_i)$, and $(y,z_i,\ldots,z_m)$ witnesses $y\close{k} z$. 
To see (6), observe that $x\close{k}y$ and
$x\close{k}z$ imply $r(y)\le r(z)\le r(x)$ or $r(z)\le r(y)\le r(x)$ in $T$, and
the conclusion follows from (5). Finally, the fact that $t$ is the tree-width of
$H$ yields $|\{z'\in X-\{z\}:r(z)\in T(z')\}|\le t$ for each $z\in X$, whence
(7) follows.
\end{proof}

\noindent We will use the properties listed in Lemma \ref{lem:close} implicitly,
without further reference.

\begin{lemma}\label{lem:low-point}
If\/ $x\le y$ in\/ $P$, then there exists\/ $z\in
X$ such that:
\begin{enumerate}
\item $x\le z\le y$,
\item $x\close{h-1}z$ and\/ $y\close{h-1}z$.
\end{enumerate}
\end{lemma}

\begin{proof}
Since $H$ contains the cover graph $G$ of $P$, there is a path
$(z_0,z_1,\dots,z_k)$ in $H$ with $z_0=x$, $z_k=y$, and $z_0<z_1<\dots<z_k$ in
$P$. Take the shortest such path. Since $h$ is the height of $P$, we have $k\le
h-1$. For each $i$ with $0\le i<k$, since $z_iz_{i+1}$ is an edge of $H$, we
have $r(z_{i+1})\in T(z_i)$ when $r(z_i)<r(z_{i+1})$ in $T$ or $r(z_i)\in
T(z_{i+1})$ when $r(z_i)>r(z_{i+1})$ in $T$. If there is an index $i$ with
$0<i<k$ and $r(z_{i-1})<r(z_i)>r(z_{i+1})$ in $T$, then we have $r(z_i)\in
T(z_{i-1})\cap T(z_{i+1})$, so $z_{i-1}z_{i+1}$ is an edge of $H$ and we can
obtain a shorter path by removing $z_i$. Therefore, there is a unique index
$i\in\{0,1,\ldots,k\}$ with $r(z_0)>r(z_1)>\dots>r(z_i)<r(z_{i+1})<\dots<r(z_k)$
in $T$. The definition of $\close{k}$ yields $z_0\close{k}z_i$ and
$z_k\close{k}z_i$. Since $k\le h-1$, the conclusion follows for $z=z_i$.
\end{proof}

\subsection{Colorings of the Ground Set}

Order the elements of $X$ as $x_0,x_1,\dots,x_{n-1}$ so that the following
holds: if $r(x_j)\le r(x_i)$ in $T$, then $j\le i$. In particular, we have $j\le
i$ whenever $x_i\close{k}x_j$. Define a coloring $\phi$ of $X$ with positive
integers using the following inductive procedure. Start by setting
$\phi(x_0)=1$. Thereafter, for $1\le i<n$, let $\phi(x_i)$ be the least positive
integer that does not belong to $\{\phi(x_j):0\le j<i$ and
$x_i\close{2h-2}x_j\}$. The reason why we take $\close{2h-2}$ in this
definition will become clear at the very end of the proof. The number of colors
used by $\phi$ is at most $1+t+t^2+\dots+t^{2h-2}$. Actually, we are not that
interested in how many colors $\phi$ will use exactly, except that this number
must be bounded in terms of $t$ and $h$, which it is.

\begin{lemma}\label{lem:coloring} 
If\/ $x\close{2h-2}z$, $x\close{2h-2}z'$ and\/ $\phi(z)=\phi(z')$, then\/ $z=z'$.
\end{lemma}

\begin{proof}
Suppose $z\ne z'$. Since $x\close{2h-2}z$ and $x\close{2h-2}z'$,
we have $z\close{2h-2}z'$ or $z'\close{2h-2}z$. Whichever of these holds, the
definition of $\phi$ yields $\phi(z)\ne\phi(z')$.
\end{proof}

Let $(x,z)$ be a pair of elements of $X$ with $x\close{2h-2}z$. There are four
cases of how $x$ and $z$ are related in $P$: (1)~$x=z$, (2)~$x<z$, (3)~$x>z$, or
(4)~$x\parallel z$. We associate with $(x,z)$ a triple
$\pval(x,z)=(\phi(x),\phi(z),t(x,z))$, where $t(x,z)$ is the number in
$\{1,2,3,4\}$ denoting which of the above four cases holds. Since the number of
distinct colors used by $\phi$ is bounded in terms of $t$ and $h$, so is the
number of distinct triples of the form $\pval(x,z)$ for all pairs $(x,z)$
considered.

We define a new coloring $\tau$ of $X$ by assigning to each element $x$ of $X$,
the family $\tau(x)=\{\pval(x,z):z\in X$ and $x\close{2h-2}z\}$. Thus
the colors used by $\tau$ are sets of triples, and the number of distinct colors
used by $\tau$ is bounded in terms of $t$ and $h$. Note that the color classes
of $\tau$ refine the color classes of $\phi$, as the first element of each
triple in $\tau(x)$ is $\phi(x)$, and $\tau(x)$ is non-empty since $x\close{0}
x$.

\begin{lemma}\label{lem:type}
If\/ $x\close{2h-2}z$, $x'\close{2h-2}z$ and\/ $\tau(x)=\tau(x')$, then:
\begin{enumerate}
\item $x\le z$ in\/ $P$ if and only if\/ $x'\le z$ in\/ $P$, 
\item $x\ge z$ in\/ $P$ if and only if\/ $x'\ge z$ in\/ $P$.
\end{enumerate}
\end{lemma}

\begin{proof}
Since $x'\close{2h-2}z$ and $\tau(x)=\tau(x')$, there is $z'\in X$
with $x\close{2h-2}z'$ and $\pval(x,z')=\pval(x',z)$. In particular, we have
$\phi(z)=\phi(z')$, which implies $z=z'$ in view of Lemma~\ref{lem:coloring}. The
conclusion now follows from $\pval(x,z)=\pval(x',z)$.
\end{proof}

\subsection{Signatures for Incomparable Pairs}

Each incomparable pair $(x,y)$ in $P$ satisfies exactly one of the following
six conditions:
\begin{enumerate}
\item $r(x)$ is below $r(y)$ in $T$,
\item $r(y)$ is below $r(x)$ in $T$, 
\item $r(x)$ is left of $r(y)$ in $T$ and $r(y')$ is left of $r(y)$ in 
$T$ for each $y'\in X$ with $\tau(y')=\tau(y)$ and $x\le y'$ in $P$,
\item $r(x)$ is left of $r(y)$ in $T$ and there exists $y'\in X$
with $\tau(y')=\tau(y)$, $x\le y'$ in $P$ and $r(y')$ not left of $r(y)$ in $T$,
\item $r(y)$ is left of $r(x)$ in $T$ and $r(x')$ is left of $r(x)$ in $T$ for
each $x'\in X$ with $\tau(x')=\tau(x)$ and $x'\le y$ in $P$,
\item $r(y)$ is
left of $r(x)$ in $T$ and there exists $x'\in X$ with $\tau(x')=\tau(x)$, 
$x'\le y$ in $P$ and $r(x')$ not left of $r(x)$ in $T$.
\end{enumerate}
We define the \textit{signature} of $(x,y)$ to be the triple
$\sigma(x,y)=(\tau(x),\tau(y),s(x,y))$, where $s(x,y)$ is the number in
$\{1,2,\dots,6\}$ denoting which of the above six cases holds for $(x,y)$. Since
the number of distinct colors used by $\tau$ is bounded in terms of $t$ and $h$,
so is the number of distinct signatures.

Let $\Inc(P,\Sigma)=\{(x,y)\in\Inc(P):\sigma(x,y)=\Sigma\}$. To finish the proof
of our main theorem, we show that $\Inc(P,\Sigma)$ is reversible for each
signature $\Sigma$. We argue by contradiction. Fix a signature $\Sigma$, and
suppose that $\Inc(P,\Sigma)$ is not reversible. In view of Lemma
\ref{lem:reversible}, $\Inc(P,\Sigma)$ contains an alternating cycle
$\{(x_i,y_i):1\le i\le k\}$. Since all the signatures $\sigma(x_i,y_i)$ are
equal, we have all the $\tau(x_i)$ equal and all the $\tau(y_i)$ equal.
Moreover, all the pairs $(x_i,y_i)$ satisfy the same one of the conditions
(1)--(6) above. This gives us six cases to consider. Case (2) is dual to (1),
(5) is dual to (3), and (6) is dual to (4). Therefore, it is enough that we show
that each of the cases (1), (3) and (4) leads to a contradiction. In the
arguments below, we always interpret the index $i$ cyclically in
$\{1,2,\dots,k\}$.

Suppose that (1) holds for all $(x_i,y_i)$. There must be an index $i$ such that
$r(x_i)$ is not below $r(x_{i-1})$ in $T$. We have $x_{i-1}\le y_i$ in $P$, so
let $z$ be an element of $X$ claimed by Lemma \ref{lem:low-point} for
$(x_{i-1},y_i)$. Thus $x_{i-1}\le z\le y_i$ in $P$, $x_{i-1}\close{h-1}z$, and
$y_i\close{h-1}z$. Since $r(x_i)$ is below $r(y_i)$ and not below $r(x_{i-1})$
in $T$, we have $r(z)\le r(x_i)<r(y_i)$ in $T$ and thus $x_i\close{h-1}z$. We
also have $\tau(x_{i-1})=\tau(x_i)$. Consequently, by Lemma \ref{lem:type}, we
have $x_i\le z\le y_i$ in $P$, which is a contradiction.

If (3) holds for all $(x_i,y_i)$, then we have $r(y_{i+1})$ left of $r(y_i)$ in
$T$ for each $i$, which is clearly a contradiction.

Finally, suppose that (4) holds for all $(x_i,y_i)$. There must be an index $i$
such that $r(x_i)$ is not left of $r(x_{i-1})$ in $T$. To simplify the notation,
let $x=x_i$, $y=y_i$ and $x'=x_{i-1}$. Thus we have $x\parallel y$ and $x'\le y$
in $P$, $\tau(x')=\tau(x)$, $r(x)$ left of $r(y)$ in $T$, and $r(x)$ not left of
$r(x')$ in $T$. Furthermore, since $(x,y)$ satisfies condition (4), there is
$y'\in X$ with $\tau(y')=\tau(y)$, $x\le y'$ in $P$ and $r(y')$ not left of
$r(y)$ in $T$. All this implies that the paths in $T$ connecting $r(x')$ to
$r(y)$ and $r(x)$ to $r(y')$ both pass through $r(x)\wedge r(y)$. Now, let $z$
be an element of $X$ claimed by Lemma \ref{lem:low-point} for $(x',y)$, and $z'$
be an element of $X$ claimed by Lemma \ref{lem:low-point} for $(x,y')$. Thus we
have $x'\le z\le y$ and $x\le z'\le y'$ in $P$, $x'\close{h-1}z$,
$y\close{h-1}z$, $x\close{h-1}z'$, and $y'\close{h-1}z'$. Since $r(z) \le r(x')$
and $r(z) \le r(y)$ in $T$, it follows that $r(z)$ is below every vertex in the
path from $r(x')$ to $r(y)$, and in particular, $r(z) \le r(x) \wedge r(y)$.
Similarly, $r(z')\le r(x)\wedge r(y)$ in $T$. Thus $r(z)\le r(z')$ or $r(z')\le
r(z)$ in $T$. If $r(z)\le r(z')$, then $z'\close{h-1}z$ and thus
$x\close{2h-2}z$. This, by Lemma \ref{lem:type}, implies $x\le z\le y$ in $P$,
which is a contradiction. If $r(z')\le r(z)$, then we get a similar
contradiction $x\le z'\le y$.  This completes the proof of Theorem~\ref{thm:main}.

\section{Questions and Problems}\label{sec:questions}

Our main result establishes the existence of the function $d(t,h)$ without
emphasis on optimizing our bound. Let $p$ be the number of colors used in
$\phi$. The number of signatures of incomparable pairs is at most $6\cdot
2^{8p^2}$. We compute $p\le 1+t+\dots+t^{2h-2} \le t^{2h-1}$, and it follows
that $d(t,h) \le 6\cdot 2^{8t^{4h-2}}$. One immediate challenge is to tighten
the bounds on this function. It may even be true that for each $t$, there is a
constant $c_t$ so that $d(t,h)\le c_th$. It is conceivable that better
techniques may prove an exact formula for $d(t,h)$, for all $t$ and $h$.

As noted in the introductory section, when the tree-width of the cover graph of
$P$ is $1$, $\dim(P)\le3$, independent of the height of $P$. Also, when the
cover graph of $P$ is outerplanar (so it has tree-width at most~$2$),
$\dim(P)\le 4$ independent of the height of $P$. On the other hand, the posets
in Kelly's construction have path-width~$3$. Accordingly, it is natural to raise
the following questions.

\begin{question}\label{que:pw=2}
Does there exist a constant $d_0$ so that if $P$ is a poset and the path-width 
of the cover graph of $P$ is at most $2$, then $\dim(P)\le d_0$?
\end{question}

\begin{question}\label{que:tw=2}
Does there exist a constant $d_1$ so that if $P$ is a poset and the tree-width 
of the cover graph of $P$ is at most $2$, then $\dim(P)\le d_1$?
\end{question}

The first of these two questions was recently settled in the affirmative
by Bir\'{o}, Keller and Young~\cite{bib:Biro-Keller-Young}, and we firmly
believe that the second one has an affirmative answer as well.

Kelly's construction actually raises two other questions.
First, is it true that a planar poset with large dimension contains
a large standard example?  We believe the answer is yes and
make the following conjecture.

\begin{conjecture}\label{con:planar-standard}
For every integer\/ $d\ge 2$, there is an integer\/ 
$D=D(d)$ so that if\/ $P$ is a planar poset with\/ $\dim(P)\ge D$,
then\/ $P$ contains the standard example\/ $S_d$.
\end{conjecture}

Second (and this specific question was posed to us by 
Stanley~\cite{bib:Stanley}), is it true that a planar poset with
large dimension has many minimal elements?  The answer is
yes. Recently, Trotter and Wang~\cite{bib:Trotter-Wang} proved
the following result.

\begin{theorem}\label{thm:Trotter-Wang}
If\/ $P$ is a planar poset with\/ $t$ minimal elements, then\/
$\dim(P)\le 2t+1$.
\end{theorem}

This inequality is best possible for $t=1$ and $t=2$, but for larger 
values of $t$, a lower bound of $t+3$ is proved in~\cite{bib:Trotter-Wang}.

The first of these two questions has a natural extension to tree-width, so we would
also make the following conjecture.

\begin{conjecture}\label{con:treewidth-standard}
For every pair\/ $(d,t)$ of positive integers with\/ $d\ge 2$, there is an integer\/ 
$D=D(d,t)$ so that if\/ $P$ is a poset such that the tree-width of
the cover graph of\/ $P$ is at most\/ $t$ and\/ $\dim(P)\ge D$,
then\/ $P$ contains the standard example\/~$S_d$.
\end{conjecture}

While the second question concerning the number of minimal elements
makes sense, it is easily answered in the negative, since adding
a zero to a poset can increase the tree-width of the cover graph by
at most one.

Finally, we close with what we believe is a very ambitious conjecture.

\begin{conjecture}\label{con:minors}
Let\/ $\cgG$ be a proper minor-closed class of graphs.  Then for every 
integer\/ $h\ge1$, there is a least positive integer\/ $d=d(\cgG,h)$ so that 
if\/ $P$ is a poset of height\/ $h$ and the cover graph of\/ $P$ belongs to\/
$\cgG$, then\/ $\dim(P)\le d$.
\end{conjecture}

Our main theorem shows that the conjecture is true when $\cgG$ is the
class of graphs of tree-width at most $t$. In~\cite{bib:Streib-Trotter}, a
general reduction is described which allows one to restrict to the case where
the cover graph has bounded diameter (as a function of the height). It follows
as an immediate corollary that the conjecture holds whenever $\cgG$ has
the diameter tree-width property.  For this reason, we have an
alternative proof of Theorem~\ref{thm:Streib-Trotter}.
Graphs of bounded genus, and more generally graphs excluding an apex graph
as a minor also have the diameter tree-width property (see~\cite{bib:Eppstein}).
Therefore, the above conjecture also holds in these special cases.

\section*{Updates}

Question \ref{que:tw=2} has been answered in the positive by Joret, Micek,
Trotter, Wang and Wiechert \cite{bib:tw=2}.
Conjecture \ref{con:minors} has been settled in the affirmative by Walczak~\cite{bib:Walczak}.

\section*{Acknowledgments}

The proof of our main theorem has been developed by the authors at a series of
meetings at conferences and campus visits, as well as through email. However, we
have also received valuable input from several other colleagues in our
respective university environments. We would also like to thank anonymous
referees who made helpful suggestions regarding the exposition and organization
of material in this paper.

The last five authors gratefully acknowledge that the first author, Gwena\"{e}l
Joret, is solely responsible for conjecturing our main theorem.

\end{document}